\documentclass{amsart}

\usepackage[a4paper,hmargin=3cm,vmargin=4cm]{geometry}
\usepackage{amsfonts,amssymb,amscd,amstext}
\usepackage{mathtools}
\mathtoolsset{showonlyrefs=true} 
\usepackage[utf8]{inputenc}
\usepackage[colorlinks=true,linkcolor=blue,citecolor=red]{hyperref}
\usepackage{braket,relsize,nccmath}

\renewcommand{\leq}{\leqslant}
\renewcommand{\geq}{\geqslant}
\newcommand{\ptl}{\partial}
\newcommand{\rr}{{\mathbb{R}}}
\newcommand{\rrn}{\mathbb{R}^{n+1}}
\newcommand{\la}{\lambda}
\newcommand{\sph}{\mathbb{S}}

\newcommand{\sub}{\subset}
\newcommand{\subeq}{\subseteq}
\newcommand{\escpr}[1]{\big<#1\big>}
\newcommand{\Sg}{\Sigma} \newcommand{\sg}{\sigma}
\newcommand{\Om}{\Omega}

\newcommand{\eps}{\varepsilon}
\newcommand{\var}{\varphi}
\newcommand{\ric}{\text{Ric}}

\DeclareMathOperator{\divv}{div}

\newtheorem{theorem}{Theorem}[section]
\newtheorem{proposition}[theorem]{Proposition}
\newtheorem{lemma}[theorem]{Lemma}
\newtheorem{corollary}[theorem]{Corollary}

\theoremstyle{definition}

\newtheorem{remark}[theorem]{Remark}
\newtheorem{remarks}[theorem]{Remarks}
\newtheorem{example}[theorem]{Example}
\newtheorem{examples}[theorem]{Examples}

\theoremstyle{remark}

\numberwithin{equation}{section}

\begin{document}

\title[hypersurfaces with boundary in weighted cylinders]{uniqueness results and enclosure properties for hypersurfaces with boundary in weighted cylinders}

\author[K.~Castro]{Katherine Castro}
\address{Katherine Castro.} 
\email{katycadi@gmail.com,ktcastro@ugr.es}

\author[C.~Rosales]{C\'esar Rosales}
\address{C\'esar Rosales. Departamento de Geometr\'{\i}a y Topolog\'{\i}a and Excellence Research Unit ``Modeling Nature'' (MNat) Universidad de Granada, E-18071, Spain.} 
\email{crosales@ugr.es}

\date{September 27, 2021}

\thanks{The second author was supported by MEC-Feder grant MTM2017-84851-C2-1-P, the research grant PID2020-118180GB-I00 funded by MCIN/AEI/10.13039/501100011033 and Junta de Andaluc\'ia grants A-FQM-441-UGR18, PY20-00164, FQM325.} 



\begin{abstract}
For a Riemannian manifold $M$, possibly with boundary, we consider the Riemannian product $M\times\rr^k$ with a smooth positive function that weights the Riemannian measures. In this work we characterize parabolic hypersurfaces with non-empty boundary and contained within certain regions of $M\times\rr^k$ with suitable weights. Our results include half-space and Bernstein-type theorems in weighted cylinders. We also generalize to this setting some classical properties about the confinement of a compact minimal hypersurface to certain regions of Euclidean space according to the position of its boundary. Finally, we show interesting situations where the statements are applied, some of them in relation to the singularities of the mean curvature flow.
\end{abstract}

\maketitle

\thispagestyle{empty}

\section{Introduction}
\label{sec:intro}
\setcounter{equation}{0}

\emph{Weighted manifolds} are Riemannian manifolds where a positive function weights the Hausdorff measures associated to the Riemannian distance. They provide a useful generalization of Riemannian geometry, with connections to many interesting topics including geometric flows and optimal transport, see Morgan~\cite[Ch.~18]{gmt} and Espinar~\cite[Sect.~4]{espinar}. 

Most of the curvature and differential operators in a Riemannian manifold have a weighted counterpart. For instance, in a manifold $M$ with smooth weight $e^\psi$ it is possible to define the \emph{drifted Laplacian} $\Delta_\psi$ and the \emph{Bakry-\'Emery-Ricci tensor} $\ric_\psi$, that extend the classical notions of the Laplacian operator and the Ricci tensor in $M$, see Section~\ref{subsec:wm}. In particular, we can introduce the \emph{parabolicity property} in the weighted setting by following the approach of Impera, Pigola and Setti in \cite{ipr} for Riemannian manifolds with non-empty boundary. More precisely, we say that $M$ is $\psi$-\emph{parabolic} if any function $u$ bounded from above such that $\Delta_\psi u\geq 0$ on $M$ and $\ptl u/\ptl\nu\geq 0$ along $\ptl M$ must be constant (here $\nu$ is the inner normal over $\ptl M$). According to the maximum principle and the Hopf lemma for the drifted Laplacian, see Proposition~\ref{prop:mp}, any compact manifold $M$ is $\psi$-parabolic. Nevertheless, the class of $\psi$-parabolic manifolds is considerably larger than the class of compact manifolds, as it is illustrated in Remark~\ref{re:vvvvvv}.

In this paper we consider a Riemannian product $M\times\rr^k$ with weight $e^\psi$. Our main objective is to characterize weighted parabolic hypersurfaces having non-empty boundary and controlled \emph{weighted mean curvature}.  The \emph{$\psi$-mean curvature} of a two-sided hypersurface is the function $H_\psi$ in \eqref{eq:fmc} defined by Gromov~\cite{gromov-GAFA} and Bayle~\cite{bayle-thesis}, who studied critical points with empty boundary of the weighted area functional. More recently, the authors  considered in \cite{castro-rosales} critical hypersurfaces for the area with free boundary. Our results in the present work classify hypersurfaces with bounded $\psi$-mean curvature confined into certain regions of $M\times\rr^k$, and entire horizontal graphs with constant $\psi$-mean curvature and free boundary in a Riemannian cylinder $M\times\rr$ under certain behaviour of the Bakry-\'Emery-Ricci tensor. As a consequence, we deduce half-space and Bernstein type theorems for hypersurfaces with non-empty boundary in some weighted manifolds. Related results for hypersurfaces with empty boundary have been established by many authors in several previous works, see for instance \cite{wang,liu,doan-bernstein,pigola-rimoldi,cheng-zhou,cmz,impera,espinar-halfspace,cavalcante-santos,espinar,delima-santos,limaetal,vieira-zhou,MR3779020,MR4077400,MR4079989,MR4194217,doan-bernstein2,hpr1,miranda-vieira}. We also generalize to the weighted context some well-known enclosure properties for compact minimal surfaces with boundary in $\rr^3$, like the convex hull property, the hyperboloid theorem and the cone theorem,  see \cite[Ch.~6]{dierkes2}. The motivation for these statements is showing that, for compact $\psi$-minimal hypersurfaces, the position of the boundary can determine the position of the whole hypersurface.

The arguments employed in our proofs rely on a simple and unified approach that applies to several situations. First, we take a function $u$ on $M\times\rr^k$ having a geometric meaning, in the sense that its level hypersurfaces are horizontal slices, vertical cylinders, round spheres,\ldots,etc. Next, we compute the weighted Laplacian $\Delta_\psi u$ and the boundary term $\ptl u/\ptl\nu$ in order to find geometric and analytical conditions ensuring that $\Delta_\psi u\geq 0$ and $\ptl u/\ptl\nu\geq 0$. From here we can deduce the results by using the maximum principle in Proposition~\ref{prop:mp} or the Liouville type property satisfied by parabolic hypersurfaces. With this scheme of work in mind we can now give a more detailed description of the paper.

In Section~\ref{sec:height} we consider the \emph{height function} $\pi(p,t):=t$ for points $(p,t)$ in a cylinder $M\times\rr$ with weight $e^\psi$. This leads to half-space results for hypersurfaces with free boundary in $\ptl M\times\rr$ or tangent to a horizontal slice $M\times\{s\}$, see Theorem~\ref{th:half-space1} and Proposition~\ref{prop:half-spaceradial}. The case of compact $\psi$-minimal hypersurfaces with boundary in a horizontal slice is also discussed, see Proposition~\ref{prop:cereza}. In Section~\ref{sec:angle} we analyze the \emph{angle function} $\theta:=\escpr{\xi,N}$, where $\xi$ is the parallel vertical vector field in $M\times\rr$ and $N$ is a unit normal to the hypersurface. As a consequence, we derive in Theorem~\ref{th:main2} a Bernstein type result for horizontal multigraphs with constant $\psi$-mean curvature and free boundary in $\ptl M\times\rr$. From here we establish in Corollaries~\ref{cor:2} and ~\ref{cor:bernstein2} some criteria ensuring uniqueness of the horizontal slices as solutions to the Bernstein problem in $M\times\rr$ with free boundary in $\ptl M\times\rr$. Finally, in Section~\ref{sec:distance} we study the \emph{squared distance function} $d(p,t):=|t|^2$ for points $(p,t)$ in a Riemannian product $M\times\rr^k$. In particular, we prove  in Theorem~\ref{th:cylinder} a classification result for hypersurfaces inside $M\times B^k(r)$ and with boundary tangent to the round cylinder $M\times\sph^{k-1}(r)$  (here $B^k(r)$ and $\sph^{k-1}(r)$ denote the round ball and sphere of radius $r>0$ in $\rr^k$ centered at the origin). As we explain in Remark~\ref{re:whole} the technique allows also to characterize hypersurfaces in $\rr^k$ with boundary tangent to a round sphere $\sph^{k-1}(r)$ or with free boundary in a smooth solid cone with vertex at the origin. Enclosure properties associated to the functions $\pi$ and $d$ are obtained in Proposition~\ref{prop:main11}, Corollary~\ref{cor:convexhull} and Proposition~\ref{prop:cylinder}. As we illustrate at the end of the paper, the arguments can be adapted to infer analogous results in other situations.

The exposition contains several examples involving interesting weights. Since our ambient manifold is the Riemannian product $M\times\rr^k$ it is natural to consider product weights $e^{h(p)}\,e^{v(t)}$. Indeed, in Lemma~\ref{lem:fproduct} we see that, in a cylinder $M\times\rr$, these weights are the unique ones for which any horizontal slice has constant weighted mean curvature. In some cases we can improve our results by assuming that the vertical component $e^{v(t)}$ is a radial or homogeneous weight in $\rr^k$. We also discuss the setting where $M\times\rr^k$ is a gradient Ricci soliton with respect to $e^\psi$, i.e., the Bakry-\'Emery-Ricci tensor $\ric_\psi$ is proportional to the metric tensor. Some perturbations of a gradient Ricci soliton are allowed, as that as the product of mixed solitons. 

As we have already mentioned, there is a connection between weighted manifolds and geometric flows. More precisely, some $\psi$-minimal hypersurfaces appear as singularities of the mean curvature flow, see Example~\ref{ex:flow} and the references therein. In relation to this, Yamamoto~\cite{yamamoto} has recently discovered that $\psi$-minimal hypersurfaces in shrinking gradient Ricci solitons model certain type I singularities of the Ricci-mean curvature flow. As consequences of our statements we provide characterization, inexistence results and enclosure properties that apply in particular for self-shrinkers, self-expanders and translating solitons with non-empty boundary in Euclidean space. 

We finish this introduction by pointing out that the method employed is also valid for hypersurfaces with empty boundary; in this case, no assumption on the boundary is needed and the Neumann condition in the definition of parabolic manifold is void, and so trivially satisfied.

\section{Preliminaries}
\label{sec:prelimi}
\setcounter{equation}{0}

In this section we introduce basic definitions and recall some results that will be used throughout the paper. The content is organized into several subsections.

\subsection{Weighted manifolds and drifted Laplacian}
\label{subsec:wm}
By a \emph{weighted manifold} we mean a connected oriented manifold $M^{n+1}$, possibly with smooth boundary $\ptl M$, together with a Riemannian metric $g:=\escpr{\cdot\,,\cdot}$, and a $C^1$ positive function $e^\psi$, that weights the Hausdorff measures in $(M,g)$. In particular, the \emph{weighted volume} of a Borel set $\Om$ and the \emph{weighted area} of a hypersurface $\Sg$ in $M$ are defined by
\[
V_\psi(\Om):=\int_{\Om}e^\psi\,dv, \quad A_\psi(\Sg):=\int_\Sg e^\psi\,da,
\]
where $dv$ and $da$ denote, respectively, the volume and area elements in $(M,g)$.

In a manifold with a $C^2$ weight $\psi$ there are generalized curvature notions involving the curvatures of $(M,g)$ and the derivatives of $\psi$. An extension of the Riemannian Ricci tensor $\ric$ is the \emph{Bakry-\'Emery-Ricci tensor} $\ric_\psi$, see \cite[p.~182]{gmt} and the references therein, which is the $2$-tensor
\begin{equation}
\label{eq:fricci}
\ric_\psi:=\ric-\nabla^2\psi,
\end{equation}
where $\nabla^2$ stands for the Hessian operator in $(M,g)$. The \emph{Bakry-\'Emery-Ricci curvature} at a point $x\in M$ in the direction of a tangent vector $w\in T_xM$ is the number $(\ric_\psi)_x(w,w)$. When we write $\ric_\psi\geq c$ for some $c\in\rr$ we mean that $(\ric_\psi)_x(w,w)\geq c\,|w|^2$ for any $x\in M$ and any $w\in T_xM$. The notation $|\cdot|$ refers to the Riemannian length of tangent vectors in $M$. If equality $\ric_\psi=c\,g$ holds then it is said that $(M,g)$ is a \emph{$c$-gradient Ricci soliton} with respect to $e^\psi$.

Most of the classical differential operators in $(M,g)$ have a weighted counterpart. For the Laplacian operator $\Delta$ in $(M,g)$ this is given by the \emph{drifted Laplacian} or \emph{$\psi$-Laplacian} $\Delta_\psi$, see \cite[Sect.~3.6]{gri-book}. This is the second order linear operator
\begin{equation}
\label{eq:flaplacian}
\Delta_\psi u:=\Delta u+\escpr{\nabla\psi,\nabla u},
\end{equation}
where $u\in C^2(M)$ and $\nabla$ denotes the gradient of functions in $(M,g)$. If we have $\Delta_\psi u\geq 0$ on $M$ then $u$ is a \emph{$\psi$-subharmonic function}.

Due to its local nature, the strong maximum principle for uniformly elliptic operators on Euclidean domains \cite[Thm.~3.5]{gilbarg-trudinger} is also satisfied for the drifted Laplacian $\Delta_\psi$ on weighted manifolds, see  \cite[Cor.~8.15]{gri-book}. There is also a weighted version of the Hopf boundary point lemma, that can be deduced as in the Euclidean case \cite[Lem.~3.4]{gilbarg-trudinger} by using a radial barrier comparison function. For future reference we state both results in the next proposition.

\begin{proposition}
\label{prop:mp}
Let $M$ be a connected Riemannian manifold with weight $e^\psi$ and $u\in C^2(M)$ a $\psi$-subharmonic function in $M\setminus\ptl M$. Then, we have:
\begin{itemize}
\item[(i)] if $u$ achieves its maximum in $M\setminus\ptl M$ then $u$ is constant,
\item[(ii)] if there is $x_0\in\ptl M$ such that $u(x)<u(x_0)$ for any $x\in M\setminus\ptl M$ then $(\ptl u/\ptl\nu)(x_0)<0$, where $\nu$ denotes the inner unit normal along $\ptl M$.
\end{itemize}
\end{proposition}

From the previous proposition, if $M$ is compact and $u\in C^2(M)$ is a $\psi$-subharmonic function with $\ptl u/\ptl\nu\geq 0$ along $\ptl M$, then $u$ is constant. Following the definition of parabolic Riemannian manifold with boundary in \cite{ips}, we say that a weighted manifold is \emph{weighted parabolic} or \emph{$\psi$-parabolic} when the following Liouville-type property holds: if a function $u\in C^2(M)$ is bounded from above, $\psi$-subharmonic in $M\setminus\ptl M$, and satisfies the Neumann condition $\ptl u/\ptl\nu\geq 0$ along $\ptl M$, then $u$ is constant. A compact manifold $M$ is always $\psi$-parabolic. 

\begin{remark}
\label{re:vvvvvv}
By using the coarea formula as in \cite[Thm.~1]{gri5}, see also \cite[Sect.~11.5]{gri-book}, we may infer parabolicity criteria for complete non-compact weighted manifolds relying on growth properties of the weighted measures. For instance, if there is $x_0\in M\setminus\ptl M$ such that $\int^{\infty} t\,V_\psi^{-1}(t)\,dt=\infty$ (resp. $\int^\infty A_\psi(t)^{-1}\,dt=\infty$), where $V_\psi(t)$ (resp. $A_\psi(t)$) is the weighted volume (resp. weighted area) of the open metric ball $B(x_0,t)$ (resp. metric sphere $S(x_0,t)$ intersected with $M\setminus\ptl M$), then $M$ is $\psi$-parabolic. So, a complete non-compact weighted manifold with $V_\psi(M)<\infty$ is $\psi$-parabolic. As an example, any smooth domain $M\subeq\rrn$ with Gaussian weight $e^{-|x|^2/2}$ is weighted parabolic.
\end{remark}

\subsection{Hypersurfaces in weighted manifolds}
\label{subsec:hypersurfaces}
Let $\Sg$ be a smooth hypersurface, possibly with smooth boundary $\ptl\Sg$, in a Riemannian manifold $(M^{n+1},g)$ with a weight $e^\psi$. The \emph{$\psi$-divergence} in $\Sg$ of a $C^1$ vector field $X$ over $\Sg$ is defined by
\[
\divv_{\Sg,\psi}X:=\divv_\Sg X+\escpr{\nabla\psi,X},
\]
where $\divv_\Sg X$ is the divergence of $X$ with respect to the induced Riemannian metric $g_{|\Sg}$. In the previous definition the term $\nabla\psi$ can be replaced with the gradient $\nabla_\Sg\psi$ of $\psi$ in $(\Sg,g_{|\Sg})$ when $X$ is tangent to $\Sg$. The \emph{$\psi$-Laplacian} in $\Sg$ is the second order linear operator
\begin{equation}
\label{eq:deltaf}
\Delta_{\Sg,\psi}u:=\divv_{\Sg,\psi}(\nabla_\Sg u)=\Delta_\Sg u
+\escpr{\nabla_\Sg\psi,\nabla_\Sg u},
\end{equation} 
where $\Delta_\Sg$ stands for the Laplacian in $(\Sg,g_{|\Sg})$. Note that $\Delta_{\Sg,\psi}$ coincides with the drifted Laplacian in \eqref{eq:flaplacian} of the Riemannian manifold $(\Sg,g_{|\Sg})$ with weight $e^\psi$. We say that $\Sg$ is \emph{$\psi$-parabolic} if such a weighted manifold is weighted parabolic.

We will always assume that $\Sg$ is a \emph{two-sided hypersurface}, so that there is a global \emph{Gauss map} on $\Sg$, i.e., a smooth unit normal vector $N$ along $\Sg$ in $(M,g)$. By following \cite[Sect.~9.4.E]{gromov-GAFA}, see also \cite[Sect.~3.4.2]{bayle-thesis}, we introduce the \emph{$\psi$-mean curvature} of $\Sg$ by means of equality
\begin{equation}
\label{eq:fmc}
H_\psi:=-\divv_{\Sg,\psi}N=nH-\escpr{\nabla\psi,N},
\end{equation}  
where $H$ is the mean curvature of $\Sg$ with respect to $N$ in $(M,g)$. We say that $\Sg$ has \emph{constant $\psi$-mean curvature} (resp. $\Sg$ is \emph{$\psi$-minimal}) if $H_\psi$ is constant (resp. $H_\psi=0$) on $\Sg$. In case $\ptl\Sg\neq\emptyset$ we say that $\Sg$ has \emph{free boundary} in $\ptl M$ if $\ptl\Sg=\Sg\cap\ptl M$ and $\Sg$ is orthogonal to $\ptl M$ along $\ptl\Sg$. It is known that $\Sg$ is a hypersurface with constant $\psi$-mean curvature and free boundary in $\ptl M$ if and only if $\Sg$ is a critical point of the weighted area functional for compactly supported variations preserving $\ptl M$ and the weighted volume, see \cite[Cor.~3.3]{castro-rosales}.

\begin{example}
\label{ex:flow}
In $\rrn$ with radial weight $e^{c |x|^2/2}$, $c\in\{-1,1\}$, a two-sided hypersurface $\Sg$ is $\psi$-minimal if and only if $nH(x)=c\,\escpr{x,N(x)}$ for any $x\in\Sg$. It was shown by Colding and Minicozzi~\cite[Lem.~2.2]{colding-minicozzi}, \cite[Sect.~1.1]{cm2} that this identity is also satisfied by the self-similar solutions to the Euclidean mean curvature flow called \emph{self-shrinkers} for $c=-1$ and \emph{self-expanders} for $c=1$. On the other hand, in $\rrn=\rr^{n}\times\rr$ with product weight $e^{\psi(p,t)}:=e^t$, the weighted minimal hypersurfaces are those for which $nH=\escpr{\xi,N}$, where $\xi$ denotes the unit vertical vector field in $\rrn$. These hypersurfaces are called \emph{translating solitons} since their evolution under the Euclidean mean curvature flow is given by translations.
\end{example}

\subsection{Weighted cylinders}
\label{subsec:wc}
For a complete oriented Riemannian manifold $M^n$, possibly with smooth boundary $\ptl M$, the \emph{Riemannian cylinder} of base $M$ is the Riemannian product $M\times\rr$, where we consider the standard metric in $\rr$. We refer the reader to the book \cite[Ch.~7]{oneill} for basic facts about the geometry of $M\times\rr$ that will be used henceforth.

As usual, the tangent space to $M\times\rr$ at a point $x=(p,t)$ is identified with $T_pM\times\rr$. We denote by $\xi$ the \emph{vertical vector field} $\xi(x):=(0,1)$. This is a unit vector field on $M\times\rr$ which is tangent along the boundary $\ptl M\times\rr$. Moreover, $\xi$ is a parallel vector field, i.e., $D_X\xi=0$ for any vector field $X$, where $D$ stands for the Levi-Civita connection in $M\times\rr$. For any $s\in\rr$, we define the \emph{horizontal slice} $M_s:=M\times\{s\}$. As $\xi$ provides a Gauss map on $M_s$, then $M_s$ is a totally geodesic hypersurface with free boundary in $\ptl M\times\rr$. 

For Riemannian cylinders it is natural to seek weights $e^\psi$ such that any horizontal slice has constant $\psi$-mean curvature. In the next lemma we show that these are the product weights, thus extending the particular case $M\subeq\rr^n$ proved in \cite[Thm.~2.1 (i)]{bcm3}.

\begin{lemma}
\label{lem:fproduct}
Let $e^\psi$ be a weight in a Riemannian cylinder $M^n\times\rr$. Then, any horizontal slice $M_s$ has constant $\psi$-mean curvature $c(s)$ if and only if there are functions $h\in C^\infty(M)$ and $v\in C^\infty(\rr)$ such that $\psi(p,t)=h(p)+v(t)$.
\end{lemma}

\begin{proof}
By equation \eqref{eq:fmc} we have $H_\psi=-\escpr{\nabla\psi,\xi}$ on $M_s$ since $M_s$ is totally geodesic and $\xi$ is a Gauss map on $M_s$. If $\psi(p,t)=h(p)+v(t)$ then $M_s$ has constant $\psi$-mean curvature $c(s):=-v'(s)$. Conversely, take a weight $e^\psi$ such that $M_s$ has constant $\psi$-mean curvature $c(s)$ for any $s\in\rr$. Given $p\in M$, let $\psi_p:\rr\to\rr$ be the smooth function $\psi_p(t):=\psi(p,t)$. Note that $\psi_p'(t)=\escpr{\nabla\psi,\xi}(p,t)=-H_\psi(p,t)=-c(t)$. Thus, $c(t)$ is a smooth function on $\rr$, and
\[
\psi(p,t)=\psi_p(t)=\psi_p(0)-\int_0^t c(s)\,ds.
\]
Hence, if we define $h(p):=\psi(p,0)$ and $v(t):=-\int_0^tc(s)\,ds$, then we get $h\in C^\infty(M)$, $v\in C^\infty(\rr)$, and $\psi(p,t)=h(p)+v(t)$, as we claimed.
\end{proof}

If $e^{\psi(p,t)}:=e^{h(p)}\,e^{v(t)}$ is a product weight on $M\times\rr$, then, for any point $x=(p,t)\in M$, and any pair of tangent vectors $(w_1,\la_1),(w_2,\la_2)$, the Bakry-\'Emery-Ricci tensor satisfies
\begin{equation}
\label{eq:riccidesc}
(\ric_\psi)_x\big((w_1,\la_1),(w_2,\la_2)\big)=(\ric_h)_p(w_1,w_2)-v''(t)\,\la_1\,\la_2,
\end{equation}
where $\ric_h$ stands for the Bakry-\'Emery-Ricci tensor in $M$ with respect to the weight $e^h$. As a consequence, if $M$ is a $c$-gradient Ricci soliton with respect to $e^h$, then $M\times\rr$ is a $c$-gradient Ricci soliton for the product weight $e^{h(p)}\,e^{-ct^2/2}$. The converse is true by a splitting result of Petersen and Wylie~\cite[Lem.~2.1]{petersen-wylie} that we can prove very shortly from Lemma~\ref{lem:fproduct}.

\begin{corollary}
\label{cor:solitons}
If the Riemannian cylinder $M^n\times\rr$ is a $c$-gradient Ricci soliton with respect to a weight $e^\psi$, then there exist $h\in C^\infty(M)$ and $v\in C^\infty(\rr)$ such that $\psi(p,t)=h(p)+v(t)$. Moreover, $M$ and $\rr$ are $c$-gradient Ricci solitons with respect to the weights $e^h$ and $e^v$, respectively.
\end{corollary}

\begin{proof}
Let us see that the $\psi$-mean curvature of any horizontal slice $M_s$ is constant. By equation \eqref{eq:fmc} we know that $H_\psi=-\escpr{\nabla\psi,\xi}$ on $M_s$. For any vector $w$ tangent to $M_s$ it is clear that
\[
w\,\big(\escpr{\nabla\psi,\xi}\big)=(\nabla^2\psi)(w,\xi)+\escpr{\nabla\psi,D_w\xi}=\ric(w,\xi)-\ric_\psi(w,\xi)=-c\,\escpr{w,\xi}=0,
\]
where we have used that $\xi$ is parallel, equation \eqref{eq:fricci}, and that $\ric(w,\xi)=0$ for any $w$. From Lemma~\ref{lem:fproduct} there are $h\in C^\infty(M)$ and $v\in C^\infty(\rr)$ such that $\psi(p,t)=h(p)+v(t)$. Finally, by equation~\eqref{eq:riccidesc} we deduce that $M$ and $\rr$ are $c$-gradient Ricci solitons with respect to $e^h$ and $e^v$.
\end{proof}

We finish this section with some definitions. By a \emph{horizontal multigraph} in $M\times\rr$ we mean a two-sided connected hypersurface $\Sg$ with a Gauss map $N$ such that the associated \emph{angle function} $\theta:=\escpr{\xi,N}$ satisfies $\theta\leq 0$ on $\Sg$. For instance, if $\Om$ is a smooth domain in $M$ and $\var\in C^\infty(\overline{\Om})$, then the \emph{horizontal graph} given by 
\[
\text{Gr}(\var):=\{(p,t)\in M\times\rr\,;\,p\in\overline{\Om},\,t=\var(p)\}
\] 
is a horizontal multigraph. Along $\text{Gr}(\var)$ we will always consider the downward unit normal 
\[
N:=\frac{(\nabla_M\var,-1)}{\sqrt{1+|\nabla_M\var|^2}},
\]
where $\nabla_M\var$ denotes the Riemannian gradient of $\var$ in $M$. When $\Om=M$ we will say that $\text{Gr}(\var)$ is an \emph{entire horizontal graph}.

\section{Analysis of the height function}
\label{sec:height}

In a Riemannian cylinder $M^n\times\rr$ the \emph{height function} is the vertical projection, i.e., the smooth function $\pi:M\times\rr\to\rr$ defined by $\pi(p,t):=t$. In this section we employ this function to prove uniqueness and enclosure results for hypersurfaces with boundary in some weighted cylinders. The reader is referred to the Introduction for an account of related works when the boundary is empty. 

We begin with a lemma where we gather some basic computations.

\begin{lemma}
\label{lem:height}
Let $M^n\times\rr$ be a Riemannian cylinder with a weight $e^\psi$. Take a two-sided hypersurface $\Sg$ in $M\times\rr$ with Gauss map $N$ and inner conormal $\nu$ along $\ptl\Sg$. Denote by $\theta:=\escpr{\xi,N}$ the associated angle function and by $H_\psi$ the $\psi$-mean curvature. Then, we have:
\begin{itemize}
\item[(i)] $\Delta_{\Sg,\psi}\pi=H_\psi\,\theta+\escpr{\nabla\psi,\xi}$ on $\Sg$,
\item[(ii)] $\ptl\pi/\ptl\nu=\escpr{\xi,\nu}$ along $\ptl\Sg$.
\end{itemize}
\end{lemma}

\begin{proof}
It is clear that $\nabla\pi=\xi$ on $M\times\rr$. Hence $\nabla_\Sg\pi=\xi-\theta N$, and so
\[
\Delta_\Sg\pi=\divv_\Sg(\nabla_\Sg\pi)=-\theta\,\divv_\Sg N=nH\,\theta,
\]
since $\xi$ is a parallel vector field and $\divv_\Sg N=-nH$. From \eqref{eq:deltaf} and \eqref{eq:fmc} we get
\[
\Delta_{\Sg,\psi}\pi=nH\,\theta+\escpr{\nabla\psi,\xi}-\escpr{\nabla\psi,N}\,\theta=H_\psi\,\theta+\escpr{\nabla\psi,\xi},
\]
which proves (i). On the other hand, note that
\[
\frac{\ptl\pi}{\ptl\nu}=\escpr{\nabla_\Sg\pi,\nu}=\escpr{\nabla\pi,\nu}=\escpr{\xi,\nu},
\]
because $\nu$ is tangent to $\Sg$. This proves (ii). 
\end{proof}

Now we obtain uniqueness results by applying the $\psi$-parabolicity condition of $\Sg$ to the height function $\pi$. For this we will consider situations where $\Delta_{\Sg,\psi}\pi\geq 0$ on $\Sg$ and $\ptl\pi/\ptl\nu\geq 0$ along $\ptl\Sg$. Note that $\pi_{|\Sg}\leq a$ if and only if $\Sg$ is contained in the lower horizontal half-space $M\times(-\infty,a]$. The fact that $\pi_{|\Sg}$ is constant is equivalent to that $\Sg$ is inside some horizontal slice $M_s:=M\times\{s\}$.

Recall that the product weights in $M\times\rr$ are the unique ones for which any slice $M_s$ has constant $\psi$-mean curvature, see Lemma~\ref{lem:fproduct}. In this setting we can prove a half-space theorem for hypersurfaces with boundary tangent to some horizontal slice or with free boundary in $\ptl M\times\rr$. 

\begin{theorem}
\label{th:half-space1}
Let $M^n\times\rr$ be a Riemannian cylinder with a product weight $e^{\psi(p,t)}:=e^{h(p)}\,e^{v(t)}$, where $h\in C^\infty(M)$ and $v\in C^\infty(\rr)$.
Suppose that there are $\la\geq0$ and $a\in\rr$ such that $v'(t)\geq\la$ for any $t\leq a$. Let $\Sg$ be a two-sided, connected and $\psi$-parabolic hypersurface contained in $M\times (-\infty,a]$. Suppose that either $|H_\psi|\leq\la$ or $\Sg$ is a horizontal multigraph with $H_\psi\leq \la$.
\begin{itemize}
\item[(i)] If $\Sg$ is tangent to $M_a$ along $\ptl\Sg$, then $\Sg\subeq M_a$ and the $\psi$-mean curvature of $M_a$ equals $\la$.
\item[(ii)] If $\Sg$ is complete and has free boundary in $\ptl M\times\rr$, then $\Sg=M_s$ for some horizontal slice of $\psi$-mean curvature $\la$.
\end{itemize}
\end{theorem}

\begin{proof}
The height function satisfies $\pi_{|\Sg}\leq a$. We have $\escpr{\nabla\psi,\xi}=v'\circ\pi\geq\la$ on $\Sg$ because $\pi_{|\Sg}\leq a$ and $v'(t)\geq\la$ for any $t\leq a$. By Lemma~\ref{lem:height} we deduce that $\pi$ is $\psi$-subharmonic on $\Sg$. This is clear when $\Sg$ is a multigraph with $H_\psi\leq\la$. Indeed
\[
\Delta_{\Sg,\psi}\pi=H_\psi\,\theta+v'\circ\pi\geq H_\psi\,\theta+\la\geq 0
\]
since $\theta\in [-1,0]$ on $\Sg$. In the case $|H_\psi|\leq\la$ on $\Sg$ we get
\[
\Delta_{\Sg,\psi}\pi=H_\psi\,\theta+v'\circ\pi\geq-|H_\psi|\,|\theta|+\la\geq 0
\]
because $|\theta|\leq 1$. On the other hand, if $\Sg$ is tangent to $M_a$ along $\ptl\Sg$ or has free boundary in $\ptl M\times\rr$, then $\ptl\pi/\ptl\nu=\escpr{\xi,\nu}=0$ along $\ptl\Sg$. In the first case this is obvious since $\xi$ is normal to $M_a$. In the second one this comes from the fact that $\nu$ equals the inner normal to $\ptl M\times\rr$ along $\ptl\Sg$ whereas $\xi$ is tangent to $\ptl M\times\rr$. Anyway, by the $\psi$-parabolicity of $\Sg$ it follows that $\pi_{|\Sg}$ is constant, i.e., $\Sg\subeq M_s$ for some $s\in\rr$. Moreover, equality $\Delta_{\Sg,\psi}\pi=0$ yields $v'(s)=\la$, which is equivalent by \eqref{eq:fmc} to that the $\psi$-mean curvature of $M_s$ with respect to $-\xi$ equals $\la$. Finally, if $\Sg$ is complete with $\ptl\Sg=\Sg\cap(\ptl M\times\rr)$, then $\Sg=M_s$. 
\end{proof}

\begin{remarks}
\label{re:opposite}
(i). The hypothesis $v'(t)\geq\la$ for any $t\leq a$ is equivalent to that $H_\psi(t)\geq\la$ for any $t\leq a$, where $H_\psi(t)$ denotes the $\psi$-mean curvature of the horizontal slice $M_t$ with respect to $-\xi$.

(ii). By replacing $\pi$ with $-\pi$ in the previous proof we see that the thesis is also true whenever $|H_\psi|\leq\la$ or $\Sg$ is a multigraph with $H_\psi\geq\la$ contained in an upper horizontal half-space $M\times[a,+\infty)$ such that $v'(t)\leq-\la$ for any $t\geq a$.

(iii). When $\la=0$ the hypotheses mean that $\Sg$ is either $\psi$-minimal or a horizontal multigraph with $H_\psi\leq 0$. The conclusion entails that the horizontal slice containing $\Sg$ is $\psi$-minimal.

(iv). The thesis fails if we do not assume some boundary condition. This is illustrated by a closed half-sphere of radius $\sqrt{n}$ about $0$ in $\rrn$ with Gaussian weight $e^{-|x|^2/2}$. 
\end{remarks}

\begin{examples}
\label{ex:morata}
When $v(t):=-c\,t^2/2$ with $c>0$ the theorem is valid for hypersurfaces with $|H_\psi|\leq\la$ contained in $M\times(-\infty,-\la/c]$ or in $M\times[\la/c,+\infty)$. This includes $\psi$-minimal hypersurfaces within the regions $t\leq0$ or $t\geq 0$. When $v(t):=\la\,t$ the result holds for hypersurfaces with $|H_\psi|\leq|\la|$ inside $M\times(-\infty,a]$ if $\la>0$, or inside $M\times[a,+\infty)$ if $\la<0$. In particular, by Corollary~\ref{cor:solitons}, the statement applies to $c$-gradient Ricci solitons of product type $M\times\rr$ with $c\geq 0$. A special case occurs when $M$ is an Einstein manifold of Ricci curvature $c>0$ and $e^{\psi(p,t)}:=e^{-ct^2/2}$. Observe that the theorem is satisfied for weighted parabolic self-shrinkers and translating solitons with boundary, as defined in Example~\ref{ex:flow}. Indeed, it implies non-existence of complete, parabolic, translating solitons within a lower horizontal half-space and with free boundary in a vertical cylinder. In relation to this, we recall that the half-space theorem for self-shrinkers and translating solitons with empty boundary was studied by Pigola and Rimoldi~\cite[Thm.~3]{pigola-rimoldi}, Cavalcante and Espinar~\cite[Thm.~1.1, Thm.~1.4]{espinar-halfspace}, and Kim and Pyo~\cite{kim-pyo}. 
\end{examples}

For compact $\psi$-minimal hypersurfaces we can reason as in the proof of Theorem~\ref{th:half-space1} to infer the following consequence.

\begin{corollary}
\label{cor:sentinel}
Let $M^n\times\rr$ be a Riemannian cylinder with a product weight $e^{\psi(p,t)}:=e^{h(p)}\,e^{v(t)}$, where $h\in C^\infty(M)$ and $v\in C^\infty(\rr)$. Suppose that $\Sg\subset M\times\rr$ is a two-sided, compact, connected, $\psi$-minimal hypersurface with free boundary in $\ptl M\times\rr$. If $v'\circ\pi$ does not change sign over $\Sg$, then $\Sg=M_s$ for some $s\in\rr$ with $v'(s)=0$. In particular, $M$ is compact.
\end{corollary}

\begin{example}
Let $M\sub\rr^n$ be a smooth region and $\Sg\sub M\times\rr$ a two-sided, compact and connected hypersurface with free boundary in $\ptl M\times\rr$. From the corollary, it follows that:
\begin{itemize}
\item[(i)] if $\Sg$ is a self-shrinker / self-expander contained in $M\times(-\infty,0]$ or $M\times[0,+\infty)$, then $M$ is compact and $\Sg=M\times\{0\}$, 
\item[(ii)] $\Sg$ cannot be a translating soliton.
\end{itemize}
\end{example}

A situation extending the unweighted setting where Theorem~\ref{th:half-space1} and Corollary~\ref{cor:sentinel} are applied is a Riemannian cylinder $M\times\rr$ with horizontal weight $e^{h(p)}$. Note that any horizontal slice $M_s$ is a minimal hypersurface for such a weight. We immediately deduce the following consequence containing half-space and Bernstein type results in this case.

\begin{corollary}
\label{cor:half-space2}
Consider a Riemannian cylinder $M^n\times\rr$ with weight $e^{\psi(p,t)}:=e^{h(p)}$, for some $h\in C^\infty(M)$. Let $\Sg$ be a two-sided and complete hypersurface with free boundary in $\ptl M\times\rr$. Suppose that one of the following conditions hold:
\begin{itemize}
\item[(i)] $\Sg$ is $\psi$-parabolic, $\psi$-minimal and contained in a lower or upper horizontal half-space,
\item[(ii)] $\Sg$ is $\psi$-minimal and compact,
\item[(iii)] $\Sg$ is a compact horizontal multigraph such that $H_\psi$ does not change sign,
\item[(iv)] $M$ is compact and $\Sg$ is an entire horizontal graph with $H_\psi$ constant. 
\end{itemize}
Then $\Sg=M_s$ for some $s\in\rr$. 
\end{corollary}

\begin{example}
The corollary applies in $M\times\rr$ for $M$ compact with constant weight. It is also valid in $\rr^n\times\rr$ with weight $e^{\psi(p,t)}:=e^{-c|p|^2/2}$ and $c\in\rr$. For the mixed Gaussian-Euclidean weight obtained when $c>0$ isoperimetric and area-minimizing hypersurfaces have been studied, respectively, by Fusco, Maggi and Pratelli~\cite{fmp}, Doan~\cite{calibrations}, and Doan and Nam~\cite{doan-bernstein}.
\end{example}

The technique employed in the proof of Theorem~\ref{th:half-space1} allows also to derive the following statement for compact $\psi$-minimal hypersurfaces with boundary in a horizontal slice.

\begin{proposition}
\label{prop:cereza}
Let $M^n\times\rr$ be a Riemannian cylinder with a product weight $e^{\psi(p,t)}:=e^{h(p)}\,e^{v(t)}$, where $h\in C^\infty(M)$ and $v\in C^\infty(\rr)$. Consider a two-sided, compact and connected $\psi$-minimal hypersurface $\Sg\sub M\times\rr$. Suppose that there is $a\in\rr$ such that one of these conditions hold:
\begin{itemize}
\item[(i)] $\Sg\subset M\times(-\infty,a]$, $v'\circ\pi\leq 0$ in $\Sg$ and $\ptl\Sg\subset M_a$,
\item[(ii)] $\Sg\subset M\times[a,+\infty)$, $v'\circ\pi\geq 0$ in $\Sg$ and $\ptl\Sg\subset M_a$.
\end{itemize}
Then $\Sg\subseteq M_a$ and $v'(a)=0$.
\end{proposition}

\begin{example}
As an application of the previous proposition, we conclude that:
\begin{itemize}
\item[(i)] in $M\times\rr$ with weight $e^{\psi(p,t)}:=e^{h(p)}$, any compact, connected $\psi$-minimal hypersurface $\Sg$ contained in a horizontal half-space $F$ and with $\ptl\Sg\subset\ptl F$ satisfies $\Sg\subset\ptl F$,
\item[(ii)] there is no compact self-shrinker $\Sg$ of $\rrn$ contained in a horizontal slab $\rr^n\times[0,a]$ with $a>0$ or $\rr^n\times [a,0]$ with $a<0$, and such that $\ptl\Sg\sub\{t=a\}$,
\item[(iii)] if $\Sg$ is a compact self-expander of $\rrn$ within a horizontal half-space $\rr^n\times(-\infty,a]$ with $a\leq 0$ or $\rr^n\times [a,+\infty)$ with $a\geq 0$, and such that $\ptl\Sg\sub\{t=a\}$, then $\Sg\subset\{t=0\}$.
\item[(iv)] there is no compact translating soliton $\Sg$ of $\rrn$ contained in a horizontal half-space $\rr^n\times[a,+\infty)$ and such that $\ptl\Sg\sub\{t=a\}$.
\end{itemize}
\end{example}

In Examples~\ref{ex:morata} we have seen that the half-space result in Theorem~\ref{th:half-space1} holds for hypersurfaces with boundary which are $\psi$-minimal for the Gaussian weight. We now show that the arguments can be extended to more general Euclidean radial weights, including the log-concave ones. Though the statement is also satisfied when $H_\psi\,\theta\geq 0$, we only consider the minimal case $H_\psi=0$.

\begin{proposition}
\label{prop:half-spaceradial}
Consider a radial weight $e^{\psi(x)}:=e^{\delta(|x|)}$ in $\rrn$ such that $\delta$ is a $C^1$ non-increasing function with $\delta'(0)=0$. Let $\Sg$ be a two-sided, connected, $\psi$-parabolic and $\psi$-minimal hypersurface contained in a closed half-space $F$ with $0\in\ptl F$. If $\escpr{n,\nu}\geq 0$ along $\ptl\Sg$, where $n$ is the unit normal to $\ptl F$ pointing outside $F$, then $\Sg$ is contained in a hyperplane parallel to $\ptl F$. Moreover, if $\psi$ is non-constant on $\Sg$, then $\Sg\subeq\ptl F$. 
\end{proposition}

\begin{proof}
We denote $P:=\ptl F$ and consider the height function $\pi_P(x):=\escpr{x,n}$ with $x\in\rrn$. Computations as in the proof of Lemma \ref{lem:height} lead to
\[
\Delta_{\Sg,\psi}\,\pi_P=H_\psi\,\theta_P+\escpr{\nabla\psi,n},
\]
where $\theta_P:=\escpr{n,N}$. Since $\Sg$ is $\psi$-minimal, $\delta'\leq 0$ and $\Sg\sub F$, we get
\[
\Delta_{\Sg,\psi}\,\pi_P=\frac{\delta'(|x|)}{|x|}\,\escpr{x,n}\geq 0, \quad\text{on }\Sg\setminus\{0\}.
\]
By continuity, it follows that $\Delta_{\Sg,\psi}\,\pi_P\geq 0$ on $\Sg$. On the other hand, we have
\[
\frac{\ptl\pi_P}{\ptl\nu}=\escpr{n,\nu}\geq 0
\]
along $\ptl\Sg$. As $\pi_{|\Sg}\leq 0$ and $\Sg$ is $\psi$-parabolic, then $\pi_P$ is constant on $\Sg$, i.e., $\Sg$ is contained in some hyperplane $P(c):=\{x\in\rrn\,;\,\escpr{x,n}=c\}$. Thus, the equality $\Delta_{\Sg,\psi}\,\pi_P=0$ becomes
\[
c\,\frac{\delta'(|x|)}{|x|}=0, \quad\text{on }\Sg\setminus\{0\},
\] 
and we conclude that $c=0$ or $\psi$ is constant on $\Sg$. This proves the claim.
\end{proof}

To finish this section we establish an enclosure property showing that a compact hypersurface must be included in any horizontal half-space that contains its boundary. For that, we will apply the maximum principle to a particular situation where $\Delta_{\Sg,\psi}\pi\geq 0$ on $\Sg$.

\begin{proposition}
\label{prop:main11}
Let $M^n\times\rr$ be a Riemannian cylinder with a product weight $e^{\psi(p,t)}:=e^{h(p)}\,e^{v(t)}$ such that $v'\geq\la\geq 0$ on $\rr$. Consider a two-sided, compact and connected hypersurface $\Sg\sub M\times\rr$ such that $|H_\psi|\leq\la$ or $\Sg$ is a horizontal multigraph with $H_\psi\leq\la$. If $\ptl\Sg\sub M\times(-\infty,a]$ for some $a\in\rr$, then $\Sg\sub M\times(-\infty,a]$. Moreover, if $\Sg$ intersects $M_a$ away from $\ptl\Sg$, or $\escpr{\xi,\nu}\geq 0$ along $\ptl\Sg$, then $\Sg\subeq M_s$ for some $s\in\rr$ with $v'(s)=\la$. 
\end{proposition}

\begin{proof}
As in the proof of Theorem~\ref{th:half-space1}, our hypotheses entail that $\Delta_{\Sg,\psi}\pi\geq 0$ on $\Sg$. Thus, the maximum principle in Proposition~\ref{prop:mp} (i) implies that $\pi_{|\Sg}$ achieves it maximum along $\ptl\Sg$. Since $\pi\leq a$ along $\ptl\Sg$ it follows that $\pi\leq a$ on $\Sg$, as desired. If $(\Sg\setminus\ptl\Sg)\cap M_a\neq\emptyset$ then $\pi_{|\Sg}$ achieves its maximum at some interior point of $\Sg$, so that $\pi_{|\Sg}=a$ on $\Sg$. Finally, if $\escpr{\xi,\nu}\geq 0$ along $\ptl\Sg$, then $\ptl\pi/\ptl\nu\geq 0$ along $\ptl\Sg$, and the $\psi$-parabolicity of $\Sg$ yields that $\pi_{|\Sg}$ is constant.
\end{proof}

\begin{example}
The proposition applies when $v(t):=\la\,t$ with $\la\geq 0$. This is the case of the translating solitons for the mean curvature flow in $\rrn$. Observe that, if $\la=0$ and $\Sg$ is $\psi$-minimal, then $\Delta_{\Sg,\psi}\pi=0$ on $\Sg$; therefore, if $\ptl\Sg$ is contained in a lower or upper horizontal half-space, then $\Sg$ is contained in the same half-space.
\end{example}

Recall that a compact minimal hypersurface $\Sg\sub\rrn$ with boundary $\ptl\Sg$ on a hyperplane $P$ coincides with the compact set enclosed by $\ptl\Sg$ in $P$. This follows because $\Sg$ satisfies the \emph{convex hull property}, i.e., $\Sg$ lies in the Euclidean convex hull of $\ptl\Sg$. In general this property fails for $\psi$-minimal hypersurfaces, as it is illustrated by a closed half-sphere of radius $\sqrt{n}$ about the origin in $\rrn$ with Gaussian weight $e^{-|x|^2/2}$. From the arguments in the proof of Proposition~\ref{prop:main11} we can obtain a convex hull property for $\psi$-minimal hypersurfaces in $M\times\rr^k$. We need some definitions. A \emph{horizontal half-space} in $M\times\rr^k$ is a set of the form $M\times F$, where $F$ is a closed half-space in $\rr^k$. The \emph{horizontal convex hull} of a non-empty compact set $S\sub M\times\rr^k$ is the intersection of all the horizontal half-spaces containing $S$. 

\begin{corollary}
\label{cor:convexhull}
Consider the Riemannian product $M^n\times\rr^k$ with horizontal weight $e^{\psi(p,t)}:=e^{h(p)}$, where $h\in C^\infty(M)$. If $\Sg$ is a two-sided, compact and connected $\psi$-minimal hypersurface, then $\Sg$ is contained in the horizontal convex hull of $\ptl\Sg$. Moreover, if $\ptl\Sg\sub M\times P$ for some hyperplane $P\sub\rr^k$, then $\Sg\sub M\times P$.
\end{corollary}

\begin{proof}
Suppose that $\ptl\Sg\sub M\times F$ for some closed half-space $F\sub\rr^k$. Let us write $F=\{t\in\rr^k\,;\,\escpr{t,n}\leq\alpha\}$, where $n$ is a unit vector in $\rr^k$ and $\alpha\in\rr$. Define the height function $\pi_F:M\times\rr^k\to\rr$ by $\pi_F(p,t):=\escpr{t,n}$. By following the computations in the proof of Lemma~\ref{lem:height} we get $\Delta_{\Sg,\psi}\,\pi_F=0$ since $\Sg$ is $\psi$-minimal and the weight does not depend on $t$. The inequality $\pi_F\leq\alpha$ on $\ptl\Sg$ and the maximum principle in Proposition~\ref{prop:mp} (i) imply that $\pi_F\leq\alpha$ on $\Sg$. This shows that $\Sg\sub M\times F$, as we wished. Moreover, if $\ptl\Sg\sub M\times\ptl F$ then we can apply the maximum principle to $\pi_F$ and $-\pi_F$ to deduce that $\Sg\sub M\times\ptl F$ as well.
\end{proof}

\section{Analysis of the angle function}
\label{sec:angle}

Let $M^n\times\rr$ be a Riemannian cylinder and $\Sg\sub M\times\rr$ a two-sided hypersurface with Gauss map $N$. Recall that the \emph{angle function} is defined by $\theta:=\escpr{\xi,N}$. If $\theta\leq 0$ then $\Sg$ is a \emph{horizontal multigraph}. We say that $\Sg$ is a \emph{cylinder} when $\theta=0$. This is equivalent to that $\xi$ is tangent to $\Sg$ and so, $\Sg$ is foliated by vertical segments. In this section we employ the function $\theta$ to establish Bernstein-type results for hypersurfaces having free boundary in $\ptl M\times\rr$ and constant $\psi$-mean curvature with respect to certain $C^2$ weights. Previous analogous results for hypersurfaces with empty boundary are mentioned in the Introduction.

We first state a lemma with some identities involving $\theta$, see \cite[Lem.~3.1]{rosales-cylinders} for a proof in a more general setting, and \cite[Lem.~3]{cavalcante-santos}, \cite[Re.~3.2]{rosales-cylinders} for related computations and references.

\begin{lemma}
\label{lem:angle}
Consider a Riemannian cylinder $M^n\times\rr$ with a weight
$e^\psi$. Let $\Sg$ be a two-sided hypersurface in $M\times\rr$ with Gauss map $N$ and inner conormal $\nu$ along $\ptl\Sg$. If $\Sg$ has constant $\psi$-mean curvature and free boundary in $\ptl M\times\rr$, then:
\begin{itemize}
\item[(i)] $\Delta_{\Sg,\psi}\,\theta+\left(\emph{Ric}_\psi(N,N)+|\sg|^2\right)\theta=\emph{Ric}_\psi(\xi,N)$ on $\Sg$,
\item[(ii)] $\ptl\theta/\ptl\nu=-\emph{II}(N,N)\,\theta$ along $\ptl\Sg$,
\end{itemize}
where $\emph{Ric}_\psi$ is the Bakry-\'Emery-Ricci tensor in \eqref{eq:fricci}, $\sg$ is the the second fundamental form of $\Sg$, and $\emph{II}$ is the second fundamental form of $\ptl M\times\rr$ with respect to the inner unit normal.
\end{lemma}

Next, we seek conditions ensuring that the horizontal slices are the unique entire horizontal graphs in $M\times\rr$ of constant $\psi$-mean curvature and free boundary in $\ptl M\times\rr$. From Lemma~\ref{lem:fproduct} we can restrict ourselves to product weights. The following statement is independent of Corollary~\ref{cor:half-space2} and shows rigidity properties for horizontal multigraphs by assuming curvature bounds on the weight and local convexity of $\ptl M$.

\begin{theorem}
\label{th:main2}
Let $M^n$ be a Riemannian manifold with locally convex boundary. In the Riemannian cylinder $M\times\rr$ we consider a product weight $
e^{\psi(p,t)}:=e^{h(p)}\,e^{v(t)}$, where $h\in C^\infty(M)$ and $v\in C^\infty(\rr)$. Suppose that $\emph{Ric}_v\leq c\leq \emph{Ric}_h$ for some constant $c\in\rr$. If $\Sg$ is a $\psi$-parabolic horizontal multigraph in $M\times\rr$ with constant $\psi$-mean curvature and free boundary in $\ptl M\times\rr$, then the angle function is constant, and we have:
\begin{itemize}
\item[(i)] $\Sg$ is a cylinder, or
\item[(ii)] $\Sg$ is contained in some horizontal slice $M_s$, or
\item[(iii)] $\Sg$ is a totally geodesic hypersurface satisfying $\emph{Ric}_\psi(N,N)=\emph{Ric}_\psi(\xi,\xi)=c$ on $\Sg$ and $\emph{II}(N,N)=0$ along $\ptl\Sg$.
\end{itemize}  
\end{theorem}

\begin{proof}
Let $N$ be any Gauss map on $\Sg$ and $\theta:=\escpr{\xi,N}$ the associated angle function. Denote by $N_h$ the horizontal component of the Gauss map, so that $N=(N_h,\theta)$. Our curvature bounds and equality~\eqref{eq:riccidesc} give us the estimates
\begin{equation}
\label{eq:adquire}
\begin{split}
c\,|N_h|^2&\leq\ric_h(N_h,N_h)=\ric_\psi\big((N_h,0),(N_h,0)\big),
\\
\ric_\psi(\xi,\xi)&=\ric_v(1,1)\leq c,
\end{split}
\end{equation}
where, for any $x=(p,t)\in\Sg$, the tensor $\ric_\psi$ is evaluated at $x$, whereas $\ric_h$ and $\ric_v$ are evaluated at $p\in M$ and $t\in\rr$, respectively. Hence, by using that $\theta^2+|N_h|^2=1$ on $\Sg$ together with the identity $\ric_\psi(w,\xi)=0$ for any horizontal vector $w$, we get
\begin{equation}
\label{eq:piki0}
\begin{split}
\ric_\psi(N,N)&=\ric_\psi\big((N_h,0),(N_h,0)\big)+\ric_\psi(\xi,\xi)\,\theta^2
\\
&=\ric_h(N_h,N_h)-\ric_v(1,1)\,|N_h|^2+\ric_v(1,1)
\\
&\geq\big(c-\ric_v(1,1)\big)\,|N_h|^2+\ric_v(1,1)
\\
&\geq\ric_v(1,1)=\ric_\psi(\xi,\xi).
\end{split}
\end{equation}

Now we choose a Gauss map $N$ on $\Sg$ for which $\theta\leq 0$. Clearly $\ric_\psi(\xi,N)=\ric_\psi(\xi,\xi)\,\theta$ on $\Sg$. From Lemma~\ref{lem:angle} (i) and inequality \eqref{eq:piki0} we deduce
\begin{equation}
\label{eq:piki1}
\ric_\psi(\xi,\xi)\,\theta=\Delta_{\Sg,\psi}\,\theta+\big(\ric_\psi(N,N)+|\sg|^2\big)\,\theta\leq\Delta_{\Sg,\psi}\,\theta+\ric_\psi(\xi,\xi)\,\theta+|\sg|^2\,\theta,
\end{equation}
and so
\begin{equation}
\label{eq:piki2}
\Delta_{\Sg,\psi}\,\theta\geq\Delta_{\Sg,\psi}\,\theta+|\sg|^2\,\theta\geq 0 \quad\text{on } \Sg.
\end{equation}
By Lemma~\ref{lem:angle} (ii) we know that
\begin{equation}
\label{eq:piki3}
\frac{\ptl\theta}{\ptl\nu}=-\text{II}(N,N)\,\theta\geq 0 \quad\text{on }\ptl\Sg,
\end{equation}
since $\ptl M$ is locally convex. Hence, the $\psi$-parabolicity of $\Sg$ implies that the angle function is a constant $\theta\in [-1,0]$. If $\theta=0$ then $\Sg$ is a cylinder. If $\theta=-1$ then $N=-\xi$ on $\Sg$.  Thus, the height function $\pi(p,t):=t$ satisfies $\nabla_\Sg\pi=\xi-\theta\,N=0$; therefore $\pi$ is constant on $\Sg$, and $\Sg\subeq M_s$ for some $s\in\rr$. Finally suppose that $\theta\in (-1,0)$, so that the horizontal projection $N_h$ does not vanish on $\Sg$. From equation \eqref{eq:piki2} we obtain $|\sg|^2=0$, i.e., $\Sg$ is totally geodesic. By having in mind \eqref{eq:piki1}, \eqref{eq:piki0} and \eqref{eq:adquire} it follows that $\ric_\psi(N,N)=\ric_\psi(\xi,\xi)=c$ on $\Sg$. Finally, from equality \eqref{eq:piki3} we conclude that $\text{II}(N,N)=0$ along $\ptl\Sg$. This completes the proof.
\end{proof}

\begin{example}
In general, under the conditions in Theorem~\ref{th:main2}, horizontal multigraphs different from horizontal slices may appear. Let $M:=\rr^{n-1}\times [0,+\infty)$. Consider the Euclidean half-space $M\times\rr$ with Gaussian weight $e^{-|x|^2/2}$. Let $\Sg$ be the intersection with $M\times\rr$ of any hyperplane in $\rr^{n+1}$ orthogonal to $\ptl M\times\rr$. From \eqref{eq:fmc} it is easy to check that $\Sg$ has constant $\psi$-mean curvature. Moreover, $\Sg$ is $\psi$-parabolic since it has finite weighted area. Clearly $\ric_\psi=1$ in $M\times\rr$ and $\text{II}=0$ on $\ptl M\times\rr$. Note also that the angle function $\theta$ can be any number in $[-1,0]$. In this setting we can show a non-totally geodesic cylinder $\Sg$ in the conditions of the statement; indeed, it suffices to take $\Sg$ as the intersection with $M\times\rr$ of the spherical cylinder $\sph^{n-1}\times\rr$.
\end{example}

In the light of the previous example it is natural to ask if Theorem \ref{th:main2} may be improved to get only horizontal slices in the thesis. In the special case of entire horizontal graphs, additional hypotheses on the base manifold $M$ or the weight $\psi$ lead to the following consequence.

\begin{corollary}
\label{cor:2}
Let $M^n$ be a Riemannian manifold with locally convex boundary. In the Riemannian cylinder $M\times\rr$ we consider a product weight $e^{\psi(p,t)}:=e^{h(p)}\,e^{v(t)}$, where $h\in C^\infty(M)$ and $v\in C^\infty(\rr)$. Suppose that $\emph{Ric}_v\leq c\leq \emph{Ric}_h$ for some constant $c\in\rr$. Consider a $\psi$-parabolic entire horizontal graph $\Sg$ in $M\times\rr$ with constant $\psi$-mean curvature and free boundary in $\ptl M\times\rr$. If one of the following hypotheses holds:
\begin{itemize}
\item[(i)] $\emph{Ric}_h(w,w)>c\,|w|^2$ for any non-zero tangent vector $w$ of $M$,
\item[(ii)] $\ptl M$ is locally strictly convex,
\item[(iii)] $h$ is constant and $c\neq 0$,
\end{itemize}
then $\Sg=M_s$ for some $s\in\rr$.
\end{corollary} 

\begin{proof}
Denote by $N$ the Gauss map on $\Sg$ for which $\theta<0$. From Theorem~\ref{th:main2} we infer that $\theta$ is constant. Clearly $\Sg$ cannot be a cylinder. In case $\Sg\subeq M_s$ for some $s\in\rr$ then $\Sg=M_s$ since $\Sg$ is an entire graph. Suppose that $\Sg$ is a totally geodesic hypersurface in $M\times\rr$ with $\ric_\psi(N,N)=\ric_\psi(\xi,\xi)=c$ on $\Sg$ and $\text{II}(N,N)=0$ along $\ptl\Sg$. 

Assume that (i) holds, and denote by $N_h$ the horizontal component of $N$. If $N_h\neq 0$ at some point of $\Sg$, we would deduce from \eqref{eq:adquire} and \eqref{eq:piki0} that $\ric_\psi(N,N)>\ric_\psi(\xi,\xi)$ at that point, a contradiction. So, we have $N=\theta\,\xi$, which implies $N=-\xi$ on $\Sg$ because $N$ and $\xi$ are unit vectors. By reasoning as in the end of the proof of Theorem~\ref{th:main2} we conclude that $\Sg=M_s$ for some $s\in\rr$. 

If hypothesis (ii) holds, then equality $\text{II}(N,N)=0$ leads to $N=\pm\xi$ along $\ptl\Sg$. Since $\theta$ is a negative constant, then $N=-\xi$ on $\Sg$, and we conclude again that $\Sg=M_s$ for some $s\in\rr$. 

Finally, suppose that $h$ is constant and $c\neq 0$. Thus $e^{\psi(p,t)}=Q\,e^{v(t)}$ for some $Q>0$. From \eqref{eq:fmc} and the fact that $\Sg$ is totally geodesic, the $\psi$-mean curvature of $\Sg$ satisfies
\[
H_\psi=-\escpr{\nabla\psi,N}=-(v'\circ\pi)\,\theta,
\]
where $\pi(p,t):=t$. Hence the function $v'\circ\pi$ is constant on $\Sg$, so that $(v''\circ\pi)\,\nabla_\Sg\pi=0$ on $\Sg$. On the other hand, for any $x=(p,t)\in\Sg$, equation~\eqref{eq:riccidesc} gives us
\[
-v''(t)=(\ric_v)_t(1,1)=(\ric_\psi)_x(\xi,\xi)=c, 
\]
so that $v''\circ\pi=-c$ on $\Sg$. Since $c\neq 0$ the equality $(v''\circ\pi)\,\nabla_\Sg\pi=0$ yields that $\pi$ is constant on $\Sg$. Therefore $\Sg=M_s$ for some $s\in\rr$. This finishes the proof.
\end{proof}

\begin{remark}
If we suppose that $\Sg$ is a horizontal multigraph instead of an entire horizontal graph then the statement is still valid with the conclusion that $\Sg\subseteq M_s$ or $\Sg$ is a cylinder. Moreover, $\Sg$ cannot be a cylinder under hypothesis (ii). On the other hand, by taking $M\times\rr$ as an unweighted half-space of $\rrn$, it is easy to see that the assumption $c\neq 0$ in hypothesis (iii) is necessary to get only horizontal slices in the thesis.
\end{remark}

For a compact base manifold $M$ we can deduce a Bernstein-type result where no extra assumptions are needed to obtain only horizontal slices in the conclusion.

\begin{corollary}
\label{cor:bernstein2}
Let $M^n$ be a compact Riemannian manifold with locally convex boundary. In $M\times\rr$ we consider a product weight
$e^{\psi(p,t)}:=e^{h(p)}\,e^{v(t)}$, where $h\in C^\infty(M)$ and $v\in C^\infty(\rr)$. Suppose that $\emph{Ric}_v\leq c\leq \emph{Ric}_h$ for some constant $c\in\rr$. If $\Sg$ is an entire horizontal graph with constant $\psi$-mean curvature and free boundary in $\ptl M\times\rr$, then $\Sg=M_s$ for some $s\in\rr$.   
\end{corollary}

\begin{proof}
Let $N$ be the Gauss map of $\Sg$ for which $\theta<0$. Note that $\Sg$ is compact since it is homeomorphic to $M$. Thus, $\Sg$ is $\psi$-parabolic and we can apply Theorem~\ref{th:main2} to get that $\theta$ is constant. Choose a point $x_0\in\Sg$ where the height function $\pi(p,t):=t$ achieves its minimum on $\Sg$. If $x_0\in\Sg\setminus\ptl\Sg$ then $N(x_0)=-\xi$, so that $\theta=-1$ on $\Sg$. If $x_0\in\ptl\Sg$, then we can write $\xi=a\,\nu(x_0)+b\,N(x_0)$. Indeed we have $\xi=b\,N(x_0)$ because $\xi\in T_{x_0}(\ptl M\times\rr)$ and $\nu(x_0)\perp T_{x_0}(\ptl M\times\rr)$ (by the free boundary condition). Since $|\xi|=|N(x_0)|=1$ and $\theta<0$, we obtain $N(x_0)=-\xi$ and $\theta=-1$ on $\Sg$. Anyway, we have proved that $\theta=-1$, and so $N=-\xi$ on $\Sg$. By reasoning as in previous results we conclude that $\Sg=M_s$ for some $s\in\rr$. 
\end{proof}

\begin{remark}
The free boundary condition along $\ptl\Sg$ cannot be removed. For instance, in a convex cylinder $M\times\rr$, where $M\sub\rr^n$ is a convex body and the weight is constant, the intersection of any non-vertical hyperplane with $M\times\rr$ provides an entire horizontal minimal graph.
\end{remark}

We finish this section by describing some interesting situations where our results are applied.

\begin{examples}
(i). By Corollary~\ref{cor:solitons}, if $M\times\rr$ is a $c$-gradient Ricci soliton with respect to a weight $e^\psi$ then $\psi(p,t)=h(p)+v(t)$, where $\ric_h=\ric_v=c$. Hence our curvature bounds in Theorem~\ref{th:main2} are satisfied. A special case occurs when $M$ is an Einstein manifold of Ricci curvature $c$ and the weight is the vertical one $e^{-ct^2/2}$. Moreover, if $c\neq 0$, then Corollary~\ref{cor:2} entails uniqueness of the horizontal slices for the associated Bernstein problem.

(ii). Our conclusions are also valid for certain perturbations of a $c$-gradient Ricci soliton. Indeed, starting from a $c$-gradient Ricci soliton $e^{h(p)}\,e^{v(t)}$ on $M\times\rr$, the curvature bounds still hold for the weight $e^{h(p)+\mu(p)}\,e^{v(t)+\rho(t)}$, where $\mu\in C^\infty(M)$ is a concave function and $\rho\in C^\infty(\rr)$ is convex. Observe that, if $\mu$ is strictly concave, then $\ric_{h+\mu}(w,w)>c\,|w|^2$ for any tangent vector $w\neq 0$ of $M$, which allows to deduce the uniqueness statement in Corollary~\ref{cor:2}.

(iii). The results in this section include the situation of a Riemannian cylinder $M\times\rr$ with a horizontal weight $e^h$ such that $\ric_h\geq c$ for some $c\geq 0$. This applies for instance when the base space is a $c$-gradient Ricci soliton with $c\geq 0$, or a log-concave perturbation of such a soliton.

(iv). Our Bernstein properties are satisfied in a Euclidean convex cylinder $M\times\rr\sub\rrn$ with some interesting product weights like $e^{\psi(x)}:=e^{-c|x|^2/2}$ with $c\in\rr$, $e^{\psi(p,t)}:=e^{-c|p|^2/2}$ with $c\geq 0$ and $e^{\psi(p,t)}:=e^{-c|p|^2/2}\,e^t$ with $c\geq 0$. In particular, we deduce Bernstein-type theorems  for self-shrinkers, self-expanders and translating solitons with boundary, see Example~\ref{ex:flow}. Note also that the restrictions in Theorem~\ref{th:main2} (iii) imply that $\Sg$ is always contained in a hyperplane orthogonal to $\ptl M\times\rr$. We finally remark that previous related results in $\rrn$ were obtained for self-shrinkers with empty boundary by Wang~\cite[Thm.~1.1]{wang}, see also Doan~\cite{doan-bernstein2}, and for entire horizontal minimal graphs for the mixed Gaussian-Euclidean weight $e^{-|p|^2/2}$ by Doan and Nam~\cite{doan-bernstein}. More recently, Hurtado, Palmer and Rosales~\cite[Ex.~4.27]{hpr1} have solved the Bernstein problem in $\rrn$ for suitable perturbations of the Gaussian weight.
\end{examples}

\section{Analysis of distance functions}
\label{sec:distance}

Let $M^n$ be a Riemannian manifold, possibly with smooth  boundary. In the Riemannian product $M\times\rr^k$ with $k\geq 2$, we take the vertical parallel vector field $\xi_i:=(0,e_i)$, where $e_i$ is the $i$th coordinate vector field in $\rr^k$ for any $i=1,\ldots, k$. The \emph{vertical projection} is the map $V:M\times\rr^k\to\rr^k$ defined by $V(p,t):=t$. Sometimes we will identify this projection with the vector field $V(p,t):=(0,t)$ on $M\times\rr^k$. The Euclidean components of $V$ are the height functions $\pi_1,\ldots,\pi_k$, where $\pi_i=\escpr{V,\xi_i}$. The (squared) \emph{vertical distance function} is given by $d:=|V|^2=\pi_1^2+\ldots+\pi_k^2$.

In this section we will use the smooth function $d$ and other combinations of the quadratic monomials $\pi_i^2$ to deduce uniqueness results and enclosure properties for hypersurfaces with non-empty boundary in $M^n\times\rr^k$ endowed with certain weights. We will need the following lemma.

\begin{lemma}
\label{lem:distance}
Consider the Riemannian product $M^n\times\rr^k$ with weight $e^\psi$. Let $\Sg$ be a two-sided hypersurface in $M\times\rr^k$ with Gauss map $N$ and inner conormal $\nu$ along $\ptl\Sg$. Denote by $\theta_i:=\escpr{\xi_i,N}$ the angle function with respect to $\xi_i$ and by $H_\psi$ the $\psi$-mean curvature of $\Sg$. Then, we have:
\begin{itemize}
\item[(i)] $\Delta_{\Sg,\psi}\pi_i^2=2\pi_i\,(H_\psi\,\theta_i+\escpr{\nabla\psi,\xi_i})+2\,(1-\theta_i^2)$, for any $i=1,\ldots,k$,
\item[(ii)] $\Delta_{\Sg,\psi}d=2\,\big(H_\psi\,\escpr{V,N}+\escpr{\nabla\psi,V}+k-1+|N_h|^2\big)$,
\item[(iii)] $\ptl d/\ptl\nu=2\,\escpr{V,\nu}$ along $\ptl\Sg$,
\end{itemize} 
where $N_h$ is the horizontal projection of $N$ in $M\times\rr^k$. 
\end{lemma}

\begin{proof}
As in the proof of Lemma~\ref{lem:height} we get $\nabla\pi_i=\xi_i$ on $M\times\rr^k$, $\nabla_\Sg\pi_i=\xi_i-\theta_i\,N$ on $\Sg$ and $\Delta_{\Sg,\psi}\pi_i=H_\psi\,\theta_i+\escpr{\nabla\psi,\xi_i}$ on $\Sg$. Hence, the identity in (i) follows from the elementary property $\Delta_{\Sg,\psi}\pi^2_i=2\pi_i\,\Delta_{\Sg,\psi}\pi_i+2\,|\nabla_\Sg\pi_i|^2$ on $\Sg$. From equality (i), the definition of $d$, and the fact that $|N_h|^2+\sum_{i=1}^k\theta_i^2=1$ on $\Sg$ we obtain (ii). Finally, note that
\[
\frac{\ptl d}{\ptl\nu}=\escpr{\nabla d,\nu}=2\,\sum_{i=1}^k\pi_i\,\escpr{\nabla\pi_i,\nu}=2\,\sum_{i=1}^k\pi_i\,\escpr{\xi_i,\nu}=2\,\escpr{V,\nu} \quad\text{in }\ptl\Sg.
\]
This completes the proof. 
\end{proof}

The weighted Laplacian $\Delta_{\Sg,\psi}d$ can be written in terms of the weighted mean curvature of the level sets of the distance function $d$. For any $r>0$, the Riemannian mean curvature of $M\times\sph^{k-1}(r)$ with respect to the unit normal $-V(p,t)/r=(0,-t/r)$ equals $\frac{k-1}{(n+k-1)r}$. Hence, the $\psi$-mean curvature function $H_{\psi,r}$ of $M\times\sph^{k-1}(r)$ satisfies
\begin{equation}
\label{eq:hcyl}
H_{\psi,r}=\frac{1}{r}\,\big(k-1+\escpr{\nabla\psi,V}\big).
\end{equation}
As a consequence $\sqrt{d}\,H_{\psi,\sqrt{d}}=\escpr{\nabla\psi,V}+k-1$ on $\Sg\setminus (M\times\{0\})$. From Lemma~\ref{lem:distance} (ii) we deduce 
\begin{equation}
\label{eq:sesost}
\Delta_{\Sg,\psi}d=2\,\big(H_\psi\,\escpr{V,N}+\sqrt{d}\,H_{\psi,\sqrt{d}}+|N_h|^2\big) \quad \text{ on } \Sg\setminus (M\times\{0\}).
\end{equation}

Now, we use the $\psi$-parabolicity condition of $\Sg$ applied to the function $d$ to prove an uniqueness result. For this we find situations where $\Delta_{\Sg,\psi}d\geq 0$ on $\Sg$ and $\ptl d/\ptl\nu\geq 0$ along $\ptl\Sg$. Observe that $d_{|\Sg}\leq r^2$ if and only if $\Sg\subset M\times B^k(r)$, where $B^k(r)\subset\rr^k$ is the closed round ball of radius $r>0$ centered at $0$. The fact that $d_{|\Sg}$ is a constant $r^2>0$ is equivalent to that $\Sg\subseteq M\times\sph^{k-1}(r)$. Having all this in mind we  deduce a cylindrical version of the half-space theorem.

\begin{theorem}
\label{th:cylinder}
Consider a Riemannian product $M^n\times\rr^k$ with weight $e^\psi$. Suppose that there exist $\la\geq 0$ and $r_0>0$ such that $H_{\psi,r}\geq\la$ for any $r\in (0,r_0]$. Let $\Sg$ be a two-sided, connected and $\psi$-parabolic hypersurface contained in $M\times B^k(r_0)$. If the $\psi$-mean curvature of $\Sg$ satisfies $|H_\psi|\leq\la$ and $\Sg$ is tangent to $M\times\sph^{k-1}(r_0)$ along $\ptl\Sg$, then $\Sg\subseteq M\times\sph^{k-1}(r_0)$ and $H_{\psi,r_0}=\la$.
\end{theorem}

\begin{proof}
We start from the expression of $\Delta_{\Sg,\psi}d$ in \eqref{eq:sesost}. Our hypotheses together with the estimates $|\escpr{V,N}|\leq\sqrt{d}$ and $|N_h|^2\geq 0$ on $\Sg$ lead to
\[
\Delta_{\Sg,\psi}d=2\,\big(H_\psi\,\escpr{V,N}+\sqrt{d}\,H_{\psi,\sqrt{d}}+|N_h|^2\big)\geq 2\,\sqrt{d}\,\big(H_{\psi,\sqrt{d}}-|H_\psi|\big)\geq 2\,\sqrt{d}\,\big(H_{\psi,\sqrt{d}}-\la\big)\geq 0.
\]
Since the previous inequality holds on $\Sg\setminus(M\times\{0\})$, which is a dense subset of $\Sg$, then $\Delta_{\Sg,\psi}d\geq 0$. Moreover, the boundary hypothesis yields $\ptl d/\ptl\nu=2\,\escpr{V,\nu}=0$ along $\ptl\Sg$. Thus, the $\psi$-parabolicity of $\Sg$ implies that $d_{|\Sg}$ is constant, and so $\Sg\subseteq M\times\sph^{k-1}(r_0)$. Finally, the fact that $H_{\psi,r_0}=\la$ follows from equality $\Delta_{\Sg,\psi}d=0$.
\end{proof}

\begin{remark}
\label{re:whole}
The statement is also valid when $M=\{0\}$, i.e., in $\rr^k$ with a weight $e^\psi$. In this case $V$ coincides with the position vector field in $\rr^k$, the horizontal projection $N_h$ vanishes and $H_{\psi,r}$ is the $\psi$-mean curvature function of the round sphere $\sph^{k-1}(r)$. Hence, the result provides a counterpart to the half-space theorem where we employ spheres as barriers instead of hyperplanes. Note that the tangency condition along $\ptl\Sg$ may be replaced with the hypothesis that $\Sg$ has free boundary in $\ptl C\setminus\{0\}$, where $C$ is a smooth solid cone in $\rr^k$ with vertex at the origin. This ensures $\escpr{V,\nu}=0$ because $V$ is tangent over $\ptl C$ and $\nu$ provides along $\ptl\Sg$ a unit normal vector to $\ptl C$. 
\end{remark}

Next, we show applications of Theorem~\ref{th:cylinder} to some interesting product weights on $M\times\rr^k$, i.e., those of the form $e^{\psi(p,t)}:=e^{h(p)}\,e^{v(t)}$ for any $(p,t)\in M\times\rr^k$. 

\begin{example}
Suppose that the vertical component is defined by a radial function $v(t):=\delta(|t|)$ where $\delta$ is $C^1$ with $\delta'(0)=0$. According to \eqref{eq:hcyl} we have
\[
H_{\psi,r}=\frac{k-1}{r}+\delta'(r),
\]
so that any cylinder $M\times\sph^{k-1}(r)$ has constant $\psi$-mean curvature. Note that $H_{\psi,r}\to+\infty$ when $r\to 0$ and so, for given $\la\geq 0$, there is $r_0>0$ such that $H_{\psi,r}\geq\la$ for any $r\in(0,r_0]$. This allows to employ the theorem for hypersurfaces inside $M\times B^k(r_0)$. On the other hand, the existence of $r>0$ for which $H_{\psi,r}=\la$ depends on the function $\delta$. For instance, this is guaranteed for any $\la\geq 0$ if $\delta'(r)\to q$ when $r\to+\infty$ with $q<0$ or $q=-\infty$. Let us discuss in detail the case $\delta(r):=\eps\,r^\alpha/\alpha$ when $\eps=\pm 1$ and $\alpha\geq 2$. 

For $\eps=-1$ the function $H_{\psi,r}$ is decreasing with respect to $r$, so that equation $H_{\psi,r}=\la$ has a unique solution $r_\la$ for given $\la\in\rr$. Thus, for any $\la\geq 0$, we can use Theorem~\ref{th:cylinder} for hypersurfaces with $|H_\psi|\leq\la$ inside $M\times B^k(r_\la)$. In particular, if a $\psi$-parabolic minimal hypersurface $\Sg\subset M\times B^k(r_0)$ with $r_0:=\sqrt[\alpha]{k-1}$ is tangent to $M\times\sph^{k-1}(r_0)$ along $\ptl\Sg$, then $\Sg\subset M\times \sph^{k-1}(r_0)$.

For $\eps=1$ the function $H_{\psi,r}$ is, as a function of $r$, strictly positive, convex, and attains its minimum for $R:=\big(\frac{k-1}{\alpha-1}\big)^{1/\alpha}$. So, we can use Theorem~\ref{th:cylinder} in any cylinder $M\times B^k(r)$ for hypersurfaces with $|H_\psi|\leq H_{\psi,R}$. Since $H_{\psi,r}>0$ for any $r>0$ we also infer a non-existence statement for $\psi$-parabolic minimal hypersurfaces with boundary inside solid cylinders $M\times B^k(r)$.

The previous analysis covers the antiGaussian and Gaussian vertical weights ($\alpha=2$ and $\eps=\pm1$). By a result of Petersen and Wylie~\cite[Lem.~2.1]{petersen-wylie} this includes those weights in $M\times\rr^k$ producing a $c$-gradient Ricci soliton with $c=\pm 1$. A special case is that of self-shrinkers and self-expanders in $\rr^{n+k}$. Self-shrinkers with empty boundary inside round balls or cylinders were studied by Vieira and Zhou~\cite[Thm.~1]{vieira-zhou}, Pigola and Rimoldi~\cite[Thm.~1]{pigola-rimoldi}, Cavalcante and Espinar~\cite[Thm.~1.2]{espinar-halfspace}, and Impera, Pigola and Rimoldi~\cite[Thm.~A]{ipr}. We remark that Theorem~\ref{th:cylinder} shows inexistence of compact self-expanders with boundary tangent to a cylinder $\rr^n\times\sph^{k-1}(r)$ or a sphere $\sph^{n+k-1}(r)$. In a similar way, there are no compact self-expanders in $\rr^{n+k}$ with free boundary in $\ptl C\setminus\{0\}$, where $C$ is a smooth solid cone with vertex at $0$.
\end{example}

\begin{example}
\label{ex:bermejo}
If $v(t)$ is constant then $H_{\psi,r}=(k-1)/r$, which is a positive and decreasing function with $H_{\psi,r}\to+\infty$ when $r\to 0$ and $H_{\psi,r}\to 0$ when $r\to+\infty$. Hence, for any $\la>0$, there is a unique $r_\la>0$ satisfying $H_{\psi,r_\la}=\la$. As $H_{\psi,r}\geq\la$ for any $r\in(0,r_\la]$ we can apply Theorem~\ref{th:cylinder} for hypersurfaces in $M\times B^k(r_\la)$ with $|H_\psi|\leq\la$. We also deduce inexistence of $\psi$-parabolic minimal hypersurfaces in $M\times B^k(r)$ with boundary tangent to $M\times\sph^{k-1}(r)$. Some particular situations are the Riemannian products $M\times\rr^k$ with constant weights and the mixed (anti)Gaussian-Euclidean weights in $\rr^{n+k}$. After a change of coordinates, the case of translating solitons in Example~\ref{ex:flow} is also included. In relation to this, P\'erez-Garc\'ia~\cite[Thm.~2.2]{perez-garcia} proved inexistence of non-compact embedded translating solitons without boundary contained in any cylinder. Other non-existence results for translating solitons with empty boundary are described in \cite[Sect.~1]{kim-pyo}.
\end{example}

\begin{example}
Suppose that $e^{v}$ is a \emph{homogeneous weight} of degree $\alpha\in\rr$. Following \cite[Sect.~3]{homostable} this means that $v$ is a $C^1$ function on $\rr^k\setminus\{0\}$ such that $e^{v(st)}=s^\alpha\,e^{v(t)}$ for any $s>0$ and $t\neq 0$. Note that $e^v$ cannot be continuously extended to $\rr^k$ as a smooth positive function. By having in mind the second equality in \cite[Lem.~3.5]{homostable} it follows that $H_{\psi,r}=(k+\alpha-1)/r$, so that $M\times\sph^{k-1}(r)$ has constant weighted mean curvature. Note that $H_{\psi,r}<0$ for any $r>0$ whenever $\alpha<1-k$. If $\alpha>1-k$ then $H_{\psi,r}$ is a positive decreasing function with $H_{\psi,r}\to+\infty$ when $r\to 0$ and $H_{\psi,r}\to 0$ when $r\to+\infty$. So, for hypersurfaces $\Sg$ with $0\notin\Sg$, we infer the same consequences as in Example~\ref{ex:bermejo}. Finally, when $\alpha=1-k$ we get $H_{\psi,r}=0$ for any $r>0$. Hence, Theorem~\ref{th:cylinder} applies for $\psi$-parabolic minimal hypersurfaces $\Sg$ contained in some cylinder $M\times\sph^{k-1}(r)$ and with $0\notin\Sg$. In particular, the only compact $\psi$-minimal hypersurfaces away from the origin that are tangent to $M\times\sph^{k-1}(r)$ along the boundary are compact regions of $M\times\sph^{k-1}(r)$. As indicated in Remark~\ref{re:whole} we can also deduce a classification statement for hypersurfaces inside a round ball $B^k(r)\sub\rr^k$ with a homogeneous weight. We remark that the result is still valid for compact hypersurfaces with free boundary in $\ptl C\setminus\{0\}$, where $C\sub\rr^k$ is a smooth solid cone with vertex at the origin.
\end{example}

Next, we employ the maximum principle to establish an enclosure property for compact hypersurfaces having non-empty boundary inside $M\times B^k(r_0)$. Similar arguments lead to an enclosure property for hypersurfaces inside round balls about the origin in $\rr^k$. In Euclidean space with constant weight we recover well-known facts for compact minimal hypersurfaces. The proposition also applies for self-expanders and translating solitons of the mean curvature flow. 

\begin{proposition}
\label{prop:cylinder}
Let $M^n\times\rr^k$ be a Riemannian product with weight $e^\psi$ such that $H_{\psi,r}\geq\la\geq 0$ for any $r>0$. Take a two-sided, compact and connected hypersurface $\Sg\subset M\times\rr^k$ with $|H_\psi|\leq\la$. If $\ptl\Sg\subset M\times B^k(r_0)$ for some $r_0>0$, then $\Sg\subseteq M\times B^{k}(r_0)$. Moreover, if $\Sg$ intersects $M\times\sph^{k-1}(r_0)$ away from $\ptl\Sg$, or $\escpr{V,\nu}\geq0$ along $\ptl\Sg$, then $\Sg\subseteq M\times\sph^{k-1}(r)$ for some $r>0$ where $H_{\psi,r}=\la$.
\end{proposition}

\begin{proof}
We can reason as in the proof of Theorem~\ref{th:cylinder} to obtain $\Delta_{\Sg,\psi}d\geq 0$. By Proposition~\ref{prop:mp} (i) the function $d_{|\Sg}$ achieves its maximum along $\ptl\Sg$. Since $d\leq r^2_0$ along $\ptl\Sg$ then $d\leq r^2_0$ on $\Sg$, as we claimed. On the other hand, if $(\Sg\setminus\ptl\Sg)\cap (M\times\sph^{k-1}(r_0))\neq\emptyset$ then the maximum of $d_{|\Sg}$ is achieved at some interior point, so that $d_{|\Sg}=r^2_0$. Finally, if $\escpr{V,\nu}\geq 0$ along $\ptl\Sg$, then $\ptl d/\ptl\nu\geq 0$ along $\ptl\Sg$ by Lemma~\ref{lem:distance} (iii), and the $\psi$-parabolicity of $\Sg$ entails that $d_{|\Sg}$ is a constant function $r>0$. From the inequalities $\la\leq H_{\psi,r}=H_\psi\leq\la$ on $\Sg$, we conclude that $H_{\psi,r}=\la$.
\end{proof}

To finish this work we show how to extend to the weighted setting other enclosure properties for compact minimal hypersurfaces with boundary in Euclidean space. More precisely, we discuss below a weighted counterpart of the ``hyperboloid theorem'', see Theorem 2 in \cite[Sect.~6.1]{dierkes2}.  

In $M^n\times\rr^{k}$ with weight $e^\psi$ we consider the function $\rho:=\sum_{i=1}^{k-1}\pi_i^2-\pi_{k}^2$. Note that $\rho^{-1}(r^2)=M\times\mathcal{H}_r$, where $\mathcal{H}_r$ is a one-sheeted hyperboloid in $\rr^{k}$ for $r\neq 0$ or a cone if $r=0$. Given a two-sided, compact and connected $\psi$-minimal hypersurface $\Sg\sub M\times\rr^{k}$ with unit normal $N$, the identity in Lemma~\ref{lem:distance} (i) and the fact that $|N_h|^2+\sum_{i=1}^{k}\theta_i^2=1$ lead to
\[
\Delta_{\Sg,\psi}\rho=2\,\bigg(\sum_{i=1}^{k-1}\pi_i\,\escpr{\nabla\psi,\xi_i}-\pi_{k}\,\escpr{\nabla\psi,\xi_{k}}\bigg)+2\,(k-3+|N_h|^2+2\theta_{k}^2).
\] 
Denote the points in $M\times\rr^{k}$ by $(p,t,s)$ with $p\in M$, $t\in\rr^{k-1}$ and $s\in\rr$. By taking a product weight $e^{\psi(p,t,s)}:=e^{h(p)}\,e^{v(t)}$, where $v(t)=\delta(|t|)$ for some $C^1$ function $\delta$ with $\delta'(0)=0$, we infer
\[
\Delta_{\Sg,\psi}\rho=2\,\delta'(|t|)\,|t|+2\,(k-3+|N_h|^2+2\theta_{k}^2).
\]
Hence $\rho_{|\Sg}$ is a $\psi$-subharmonic function when $\delta$ is non-decreasing and $k\geq 3$. By the maximum principle in Proposition~\ref{prop:mp} (i) we deduce that $\rho_{|\Sg}$ attains its maximum along $\ptl\Sg$. Hence, if $\ptl\Sg$ is inside a region $\{\rho\leq r^2\}$ for some $r\geq 0$, then the same holds for $\Sg$. Moreover, in the case $r=0$, the regularity and connectedness of $\Sg$ entail that $\ptl\Sg$ cannot intersect both cones $\{\rho>0\}$ and $\{\rho<0\}$. This generalizes the ``cone theorem'' for minimal surfaces in Euclidean space, see Theorem 3 in \cite[Sect.~6.1]{dierkes2}. Similar conclusions are achieved when $e^v$ is a homogeneous weight of non-negative degree in $\rr^{k-1}$. The analysis is also valid when $M=\{0\}$; in particular, it applies for compact translating solitons.

\begin{remark}
We can employ the same arguments with the function $\mu:=\sum_{i=1}^{k-1}\pi_i^2-\pi_k$ to deduce a ``paraboloid theorem'' for compact $\psi$-minimal hypersurfaces with boundary in $M\times\rr^k$.
\end{remark}

\providecommand{\bysame}{\leavevmode\hbox to3em{\hrulefill}\thinspace}
\providecommand{\MR}{\relax\ifhmode\unskip\space\fi MR }
\providecommand{\MRhref}[2]{%
  \href{http://www.ams.org/mathscinet-getitem?mr=#1}{#2}
}
\providecommand{\href}[2]{#2}

\end{document}